\documentclass{amsart}[11 pt]

\usepackage{blindtext}
\usepackage[margin=1.5in]{geometry}
\usepackage[dvipsnames]{xcolor}
\usepackage{tikz-cd}
\usepackage[utf8]{inputenc}
\usepackage[english]{babel}
\usepackage{graphicx}
\usepackage{amsmath,amssymb}

\usepackage{amsthm}
\usepackage{color}
\usepackage[T1]{fontenc}
\usepackage{etoolbox}
\usepackage{mathtools}

\usepackage{mathtools}
\usepackage{blindtext}
\usepackage{hyperref}
\usepackage{enumerate} 

\usepackage{accents}

\newtheorem{thm}{Theorem}[section]
\newtheorem{cor}[thm]{Corollary}
\newtheorem{lem}[thm]{Lemma}

\theoremstyle{definition}
\newtheorem{defn}[thm]{Definition}
\theoremstyle{remark}

\numberwithin{equation}{section}

\theoremstyle{definition}

\newcommand{\calU}{\mathcal{U}}

\newcommand{\calZ}{\mathcal{Z}}

\newcommand{\calV}{\mathcal{V}}

\makeatletter
\def\bign#1{\mathclose{\hbox{$\left#1\vbox to8.5\p@{}\right.\n@space$}}\mathopen{}}
\makeatother

\DeclareMathOperator{\Ima}{Im}

\usepackage{graphicx}

\let\svthefootnote\thefootnote
\newcommand\freefootnote[1]{%
  \let\thefootnote\relax%
  \footnotetext{#1}%
  \let\thefootnote\svthefootnote%
}

\newcommand\nnfootnote[1]{%
  \begin{NoHyper}
  \renewcommand\thefootnote{}\footnote{#1}%
  \addtocounter{footnote}{-1}%
  \end{NoHyper}
}

\hypersetup{
    colorlinks,
    citecolor=black,
    filecolor=black,
    linkcolor=black,
    urlcolor=black
}

\begin{document}  
\title{$K$-stability of $\calZ$-stable C*-algebras} 
\author{Shanshan Hua} 
\begin{abstract} 
We provide a shorter new proof of the fact that $\calZ$-stable C*-algebras are $K_{1}$-surjective using the R{\o}rdam-Winter picture of the Jiang-Su algebra $\calZ$. Consequently, we recapture the $K$-stability of $\calZ$-stable C*-algebras. 
\end{abstract} 
\maketitle 

\nnfootnote{This work forms part of the author's DPhil thesis at the University of Oxford.} 

\section{Introduction} 
Operator algebraic $K$-theory is defined using equivalence classes of unitaries and projections in matrix amplifications over a C*-algebra. By Bott periodicity, the $K$-theory of a unital C*-algebra $A$ is given by the higher homotopy groups of $\calU_{\infty}(A)$, where $\pi_{2n} (\calU_{\infty}(A)) \cong K_{1}(A)$ and $\pi_{2n+1} (\calU_{\infty}(A)) \cong K_{0}(A)$ for $n\in \mathbb{N}$, see \cite[Proposition 11.4.1]{Rordam-K} for instance. On the other hand, non-stable $K$-theory studies the homology theory given by the higher homotopy groups of the unitary group $\calU(A)$ based at $1_{A}$. The $K$-theory of stable C*-algebras can be computed without taking matrix amplications and thus coincides with their non-stable $K$-theory. However, $K$-theory and non-stable $K$-theory are different in general, even in the case of commutative C*-algebras. Take $C(S^{3})$ as an example, where we have $K_{1}(C(S^{3})) = \mathbb{Z}$, but $\pi_{0}(\calU(C(S^{3}))) = 0$. 

In many important cases, even without stability, the $K$-theory and non-stable $K$-theory agree, see \cite{Rieffel_tori} for noncommutative irrational tori and \cite{RR0_Zhang} for simple C*-algebras with real rank zero and stable rank one. In general, exact characterisations of when stable and non-stable $K$-theory coincide are not known, even at the level of $\pi_{0}$. Blackadar raised a question in \cite{Blackadar_nonstable_K} of when a C*-algebra is $K_{1}$-injective or $K_{1}$-surjective, namely when the canonical homomorphism from $\pi_{0}(\calU(A))$ to $K_{1}(A)$ is injective or surjective. A stronger notion called $K$-stability is defined by Thomsen in \cite{Thomsen}, which means that for any $n\in \mathbb{N}$, the canonical inclusion $\calU(M_{m}\otimes A)\rightarrow \calU_{\infty}(A)$ induces isomorphisms between higher homotopy groups of $\calU(M_{m}\otimes A)$ and the $K$-theory of $A$. C*-algebras are $K$-stable under certain structural conditions, such as having stable rank one in \cite{Rieffel_tori} or being simple and purely infinite in \cite{RR0_Zhang}. However, it remains open whether simple C*-algebras are automatically $K$-stable or $K_{1}$-isomorphic. 

Recently, non-stable $K$-theory has shown its importance in the new abstract approach of classification for C*-algebras initiated in \cite{reproof_TWW}. In the history of the Elliot classification program, a central object is the Jiang-Su algebra $\calZ$, which was originally constructed in \cite{Jiang-Su} as the inductive limit of dimension drop algebras, with explicitly defined connecting maps. Tensorial absorption of $\calZ$ excludes the exotic C*-algebras constructed in \cite{Villadsen}, \cite{Rordam-counter}, \cite{Toms-counter} and is needed to obtain a satisfying classification theorem for simple separable nuclear C*-algebras by $K$-theory and traces. Though this is a different tensorial absorption property from stability, it does imply similar nice $K$-theoretical properties. For simple $\calZ$-stable exact C*-algebras, a dichotomy in \cite{Dichotomy} shows that they are either purely infinite or stably finite. R{\o}rdam proved in 2004 in \cite{SR_RR_Z} that simple stably finite $\calZ$-stable C*-algebras necessarily have stable rank one, by extensively studying structural properties of $\calZ$-stable C*-algebras. Then combining results in \cite{RR0_Zhang} and \cite{Rieffel_tori}, one can conclude that $\calZ$-stability implies $K$-stability for simple C*-algebras. 

The result that $\calZ$-stable C*-algebras are $K$-stable in the general setting is actually proven ealier by Jiang in his unpublished paper \cite{Jiang}, which in particular implies $K_{1}$-injectivity and $K_{1}$-surjectivity for $\calZ$-stable C*-algebras. In the new abstract approach to the classification theorem in \cite{big_classification}, $K_{1}$-injectivity for certain non-simple $\calZ$-stable C*-algebras plays an important role in \cite[Section 5.4]{big_classification} to obtain the uniqueness theorem for maps. Since the result in \cite{Jiang} remains unpublished, a short and self-contained new proof for $K_{1}$-injectivity of $\calZ$-stable C*-algebras was provided in \cite[Section 4.2]{big_classification}. The proof takes advantage of a modern description of $\calZ$, given by R{\o}rdam and Winter in \cite{new-pic}, around ten years after its introduction. They realize $\calZ$ as a stationary inductive limit of the so-called generalized dimension drop algebras, whose fibres are UHF-algebras instead of finite matrix algebras. This should be seen as an alternative definition for $\calZ$, since many fundamental properties of $\calZ$ can be readily obtained from this viewpoint, for instance in reproving strong self-absorption of $\calZ$ in \cite{Andre} and Winter's localization of the Elliott conjecture in \cite{localize}, which is vital in the Gong-Lin-Niu theorem \cite{GLN}.   

In this paper, we will recapture the rest of Jiang's $K$-stability results. Jiang's proof relies on explicit constructions of homomorphisms to produce inverse maps between higher homotopy groups when proving $K$-stability, which then implies $K_{1}$-injectivity and $K_{1}$-surjectivity of $\calZ$-stable C*-algebras. In comparison, we first provide a short new proof for $K_{1}$-surjectivity of $\calZ$-stable C*-algebras in Section 3, using the R{\o}rdam-Winter picture in the spirit of the $K_{1}$-injectivity proof of \cite{big_classification}. The extra flexibility is given by UHF-fibers of generalized dimension drop algebras, since $M_{n^{\infty}}$-stable C*-algebras are shown to have nice $K$-theoretical properties in Lemma \ref{UHF-surjective} and \cite[Lemma 4.9]{big_classification}, which is not true in general when tensoring by finite matrix algebras. With $K_{1}$-injectivity and $K_{1}$-surjectivity of $\calZ$-stable C*-algebras, we are able to recapture $K$-stability in Theorem \ref{main}, without needing to construct explicit maps. 

The result in the paper is becoming increasingly important. For instance, in Toms' recent work \cite{homotopy_cuntz} on higher homotopy groups of Cuntz classes for simple approximately divisible C*-algebras of real rank zero, $K_{1}$-surjectivity plays an important role, and it is likely that this will continue to be the case in further work in this direction. 

\section{Preliminaries} 
For a unital C*-algebra $A$, we will denote the set of untaries in $A$ by $\calU(A)$. For $n\geq 1$, we will use the notation $1_{n}$ for the unit of $M_{n}\otimes A$ and $\calU_{n}(A)$ for the set of unitaries in $M_{n}\otimes A$. The set of unitaries in $\calU_{n}(A)$ that are homotopic to $1_{n}$ is denoted by $\calU^{0}_{n}(A)$. If $A$ is a non-unital C*-algebra, we denote by $\tilde{A}$ the unitization of $A$, with canonically associated maps $\iota: A\rightarrow\tilde{A}$ and $q: \tilde{A}\rightarrow \mathbb{C}$. 

\begin{defn} \label{defn_K} 
A unital C*-algebra $A$ is called \textit{$K_{1}$-surjective ($K_{1}$-injective)} if the group homomorphism $\calU(A)/\calU^{0}(A)\rightarrow K_{1}(A)$, $[u]\mapsto [u]_{1}$, is surjective (injective). 

In the non-unital case, a C*-algebra $A$ is called \textit{$K_{1}$-surjective ($K_{1}$-injective)} if $\tilde{A}$ is $K_{1}$-surjective ($K_{1}$-injective). 
\end{defn} 

We will also describe briefly the different descriptions of $\calZ$. For coprime integers $p, q$, the \textit{prime dimension drop algebra} is defined to be 
\begin{equation} 
\calZ_{p,q} = \{f\in C([0,1], M_{p}\otimes M_{q}): f(0)\in M_{p}\otimes 1_{q}, f(1) \in 1_{p}\otimes M_{q}\}. 
\end{equation} 
The Jiang-Su algebra $\calZ$ was originally constructed in \cite{Jiang-Su} as an inductive limit of such $\calZ_{p,q}$, together with delicately defined connecting maps. In fact, it is the unique inductive limit of prime dimension drop algebras that is unital, simple, monotracial and $K_{\ast}(A)\cong K_{\ast} (\mathbb{C})$ by \cite[Theorem 6.2]{Jiang-Su}. 

The R{\o}rdam-Winter picture of $\calZ$ uses the \textit{generalized prime dimension drop algebra} corresponding to infinite coprime supernatural numbers $\mathfrak{p}$ and $\mathfrak{q}$, which is defined by 
\begin{equation}
\calZ_{\mathfrak{p}, \mathfrak{q}}=\big\{f\in C([0,1], M_{\mathfrak{p}}\otimes M_{\mathfrak{q}}):  f(0)\in M_{\mathfrak{p}}\otimes 1_{\mathfrak{q}}, f(1)\in 1_{\mathfrak{p}}\otimes M_{\mathfrak{q}}\big\}. 
\end{equation} 

A unital endomorphism $\varphi$ on a unital C*-algebra $A$ is said to be \textit{trace-collapsing} if $\tau \circ \varphi = \tau'\circ \varphi$ for any $\tau, \tau'\in T(A)$. Then the following theorem is proven. 

\begin{thm} \label{new-pic} 
(\cite[Theorem 3.4]{new-pic}) Let $\mathfrak{p}, \mathfrak{q}$ be infinite coprime supernatural numbers. Then there exists a trace-collapsing unital endomorphism $\varphi$ on $\calZ_{\mathfrak{p}, \mathfrak{q}}$. For any such $\varphi$, the Jiang-Su algebra $\calZ$ is isomorphic to the stationary inductive limit of the sequence 
\begin{equation*}
\begin{tikzcd}
\calZ_{\mathfrak{p}, \mathfrak{q}} \arrow[r, "\varphi"] & \calZ_{\mathfrak{p}, \mathfrak{q}} \arrow[r, "\varphi"] & \calZ_{\mathfrak{p}, \mathfrak{q}} \arrow[r, "\varphi"] & \cdots. 
\end{tikzcd}
\end{equation*} 
\end{thm} 

As noted in the introduction, this viewpoint for $\calZ$ is used to reprove $K_{1}$-injectivity of unital $\calZ$-stable C*-algebras in \cite[Theorem 4.8]{big_classification}. This quickly extends to the non-unital case, which recaptures Jiang's $K_{1}$-injectivity result in \cite[Theorem 2]{Jiang}. 

\begin{thm} \label{K1-inj} 
If $A$ is a C*-algebra, then $A\otimes \calZ$ is $K_{1}$-injective. 
\end{thm} 

\begin{proof} 
The unital case is proven in \cite[Theorem 4.8]{big_classification}. When $A$ is non-unital, take $u\in \calU((A\otimes \calZ)^{\sim})$ with $[u]_{1} = 0\in K_{1}(A\otimes \calZ)$. Then $[u]_{1} = 0\in K_{1}(\tilde{A}\otimes \calZ)$ and $(q\otimes \mathrm{id}_{\calZ}(u)) = \lambda 1_{\calZ}$ for some $\lambda\in \mathbb{T}$. Since $\tilde{A}\otimes \calZ$ is $K_{1}$-injective, there is a path of unitaries $(u_{t})_{t\in [0,1]} \subseteq \calU(\tilde{A}\otimes \calZ)$ such that $u_{0} = 1_{\tilde{A}\otimes \calZ}$ and $u_{1} = u$. For $t\in [0,1]$, take $v_{t} = (1_{\tilde{A}}\otimes (q\otimes \mathrm{id}_{\calZ})(u_{t}))^{\ast} u_{t}$, then $(v_{t})_{t\in [0,1]}$ is a path of unitaries in $\calU(\tilde{A}\otimes \calZ)$ with $v_{0} = 1_{\tilde{A}\otimes \calZ}$ and $v_{1} = \overline{\lambda} u$. Since $(q\otimes \mathrm{id}_{\calZ})(v_{t}) = 1_{\calZ}$ for $t\in [0,1]$, then $(v_{t})_{t\in [0,1]}\subseteq \calU((A\otimes \calZ)^{\sim})$, which implies $1_{(A\otimes \calZ)^{\sim}}\sim_{h} \overline{\lambda} u\sim_{h} u$ in $(A\otimes \calZ)^{\sim}$. 
\end{proof} 

Since the approach in \cite[Theorem 4.8]{big_classification} inspires the results of this paper, we will use Theorem \ref{K1-inj} freely in later sections. 

\section{\texorpdfstring{$K_{1}$-surjectivity of $\calZ$-stable } \ C*-algebras}
We start by proving $K_{1}$-surjectivity of UHF-absorbing C*-algebras. 

\begin{lem} \label{UHF-surjective}
If $A$ is a unital C*-algebra and $n\geq 2$, then $A\otimes M_{n^{\infty}}$ is $K_{1}$-surjective. 
\end{lem}  

\begin{proof} 
Take $B = A \otimes M_{n^{\infty}}$ and $[u]_{1}\in K_{1}(B)$, where $u\in \calU_{m}(B)$ for some $m\in\mathbb{N}$. By enlarging $m$ if necessary, we may assume that $m=n^{k}$ for some $k\in\mathbb{N}$. Since there exists an an isomorphism between $M_{n^{\infty}}$ and $M_{n^{\infty}}\otimes M_{n^{k}}$ that is approximately unitarily equivalent to the first-factor embedding, there exists a $\ast$-isomorphism $\theta: B \rightarrow B\otimes M_{n^{k}}$ such that $\theta$ is approximately unitarily equivalent to $\text{id}_{B} \otimes 1_{n^{k}}$.  

Take $v = \theta^{-1}(u)\in \calU(B)$, then there exists $w\in \calU_{n^{k}} (B)$ such that 
\begin{equation} 
\|u - w(v\otimes 1^{n^{k}}) w^{\ast}\| = \|\theta(v) - w(v\otimes 1^{n^{k}}) w^{\ast}\| < 1. 
\end{equation} 
Then $u$ is homotopic to $w(v\otimes 1^{n^{k}}) w^{\ast}$ in $\calU_{n^{k}}(B)$ and thus $[u]_{1}=[v^{n^{k}}]_{1}\in K_{1}(B)$. 
\end{proof} 

Lemma \ref{UHF-surjective} extends to non-unital UHF-absorbing C*-algebras by a direct application of the following lemma, which will also be applied to non-unital $\calZ$-stable C*-algebras later. The proof is standard and we include it for completeness. 

\begin{lem} \label{non-unital-K1} 
Let $A, B$ be C*-algebras with $A$ non-unital and $B$ unital. If $\tilde{A}\otimes B$ is $K_{1}$-surjective and $B$ is $K_{1}$-injective, then $A\otimes B$ is $K_{1}$-surjective, where $\otimes$ is the minimal tensor product. 
\end{lem} 

\begin{proof} 
By \cite[Lemma 1]{exact-short} for example, we have the following split short exact sequence, 
\begin{equation} \label{unit-tensor} 
\begin{tikzcd}
0 \arrow[r] & A\otimes B \arrow[r, "\iota\otimes \mathrm{id}_{B}"] & \tilde{A}\otimes B \arrow[r, "q \otimes \mathrm{id}_{B}"] & B \arrow[r] & 0, 
\end{tikzcd}
\end{equation} 
which induces a split short exact sequence for $K_{1}$-groups by  split-exactness of $K_{1}$, see \cite[Proposition 8.2.5]{Rordam-K}. Then $(\iota\otimes \mathrm{id}_{B})^{\sim}$ is injective and for any $u\in \calU(\tilde{A}\otimes B)$, it is in $\calU((A\otimes B)^{\sim})$ if and only if $(q\otimes \mathrm{id_{B}})(u)\in \mathbb{C} 1_{B}$. 

Take any $x\in K_{1}(A \otimes B) \subseteq K_{1}(\tilde{A}\otimes B)$, then there exists $v_{0}\in \calU(\tilde{A}\otimes B)$ such that $x=[v_{0}]_{1}$ since $\tilde{A}\otimes B$ is assumed to be $K_{1}$-surjective. Then 
\begin{equation} 
[(q\otimes \mathrm{id}_{B})(v_{0})]_{1} = K_{1}(q\otimes \mathrm{id}_{B})(x) = 0 \in K_{1}(B), 
\end{equation} 
which implies $(q\otimes \mathrm{id}_{B})(v_{0})\in \calU^{0}(B)$ by $K_{1}$-injectivity of $B$. Therefore there exists $v_{1}\in \calU^{0}(\tilde{A}\otimes B)$ with $(q\otimes \mathrm{id}_{B})(v_{0}) = (q\otimes \mathrm{id}_{B})(v_{1})$. We take $u = v_{1}^{*}v_{0} \in \calU(\tilde{A}\otimes B)$ so that $[u]_{1} = x$ and $(q\otimes \mathrm{id}_{B})(u) = 1_{B}$, which implies that $u\in \calU((A\otimes B)^{\sim})$. 
\end{proof} 

\begin{cor} \label{UHF-non-unital} 
If $A$ is a C*-algebra and $n\geq 2$, then $A\otimes M_{n^{\infty}}$ is $K_{1}$-surjective. 
\end{cor} 

\begin{proof} 
In the unital case, the result follows from Lemma \ref{UHF-surjective}. If $A$ is non-unital, since $\tilde{A}\otimes M_{n^{\infty}}$ is $K_{1}$-surjective and $M_{n^{\infty}}$ is unital and $K_{1}$-injective, by Lemma \ref{non-unital-K1}, we get $K_{1}$-surjectivity of $A\otimes M_{n^{\infty}}$. 
\end{proof} 

Now we prove the following main result of the section. We will denote the \textit{suspension} of a C*-algebra $A$ by $SA = C_{0}((0,1))\otimes A$.  

\begin{lem} \label{K1-surjective}
If $A$ is a unital C*-algebra, then $A\otimes \calZ_{2^{\infty}, 3^{\infty}}$ is $K_{1}$-surjective.  
\end{lem}  

\begin{proof} 
We view $A\otimes \calZ_{2^{\infty}, 3^{\infty}}$ as a C*-subalgebra of $C([0,1], A\otimes M_{6^{\infty}})$, where 
\begin{equation} 
A\otimes \calZ_{2^{\infty}, 3^{\infty}} = \{f\in C([0,1], A\otimes M_{6^{\infty}}): f(0) \in A\otimes M_{2^{\infty}}, f(1) \in A\otimes M_{3^{\infty}}\}. 
\end{equation} 
Then we have the following short exact sequence, 
\begin{equation}
0\rightarrow A\otimes SM_{6^{\infty}}\xrightarrow{\iota} A\otimes \calZ_{2^{\infty}, 3^{\infty}} \xrightarrow{q} (A\otimes M_{2^{\infty}})\oplus (A\otimes M_{3^{\infty}})\rightarrow 0,
\end{equation} 
where $q(f)=(f(0), f(1))$ for any $f\in A\otimes \calZ_{2^{\infty}, 3^{\infty}}$. 

Take $u\in \calU_{m}(A\otimes \calZ_{2^{\infty}, 3^{\infty}})$ for some $m\in \mathbb{N}$ so that $u(0)\in \calU_{m}(A\otimes M_{2^{\infty}})$ and $u(1)\in \calU_{m}(A\otimes M_{3^{\infty}})$. By Theorem \ref{UHF-surjective}, both $A\otimes M_{2^{\infty}}$ and $A\otimes M_{3^{\infty}}$ are $K_{1}$-surjective. Then there exists $w_{0}\in \calU(A\otimes M_{2^{\infty}})$ and $w_{1}\in \calU(A\otimes M_{3^{\infty}})$ such that $[u(0)]_{1}=[w_{0}]_{1}\in K_{1}(A\otimes M_{2^{\infty}})$ and $[u(1)]_{1}=[w_{1}]_{1}\in K_{1}(A\otimes M_{3^{\infty}})$. 

Since $u$ is a continuous path in $\calU_{m}(A\otimes M_{6^{\infty}})$ connecting $u(0)$ and $u(1)$, then 
\begin{equation} 
[w_{0}]_{1} = [u(0)]_{1} = [u(1)]_{1} = [w_{1}]_{1} \in K_{1}(A\otimes M_{6^{\infty}}), 
\end{equation} 
where we regard $A\otimes M_{2^{\infty}}$ and $A\otimes M_{3^{\infty}}$ as C*-subalgebras of $A\otimes M_{6^{\infty}}$. Since $A\otimes M_{6^{\infty}}$ is unital and UHF-absorbing, then it is $K_{1}$-injective by \cite[Lemma 4.9]{big_classification}. Thus there exists a unitary path $w\in C([0,1])\otimes A\otimes M_{6^{\infty}}$ with $w(0)=w_{0}\in A\otimes M_{2^{\infty}}$ and $w(1)=w_{1}\in A\otimes M_{3^{\infty}}$. Thus $w\in \calU(A\otimes \calZ_{2^{\infty}, 3^{\infty}})$ with $q(w)=(w_{0},w_{1})$ and we have 
\begin{equation} 
[q(u)]_{1} = [q(w)]_{1}\in K_{1}(A\otimes M_{2^{\infty}})\oplus K_{1}(A\otimes M_{3^{\infty}}). 
\end{equation} 
Then $[u(w\oplus 1_{m-1})^{*}]_{1}\in \ker(K_{1}(q))= \Ima(K_{1}(\iota))$ by half-exactness of $K_{1}$, see \cite[Proposition 8.2.4]{Rordam-K}. Since $A\otimes SM_{6^{\infty}}$ is $K_{1}$-surjective by Corollary \ref{UHF-non-unital}, there exists $z\in \calU((A\otimes SM_{6^{\infty}})^{\sim})\subseteq A\otimes \calZ_{2^{\infty}, 3^{\infty}}$ such that $[u(w\oplus 1_{m-1})^{*}]_{1}=K_{1}(\iota)([z]_{1})$. Thus $[u]_{1}=[zw]_{1}\in K_{1}(A\otimes \calZ_{2^{\infty}, 3^{\infty}})$. 
\end{proof} 

For infinite coprime supernatural numbers $\mathfrak{p}, \mathfrak{q}$, then $\calZ_{\mathfrak{p}, \mathfrak{q}}$ is $\calZ$-stable by \cite[Theorem 4.6]{fiber-alg}, since fibers of $\calZ_{\mathfrak{p}, \mathfrak{q}}$ are $\calZ$-stable. Then a direct application of Jiang’s result shows that $A\otimes \calZ_{2^{\infty}, 3^{\infty}}$ is $K_{1}$-surjective. Whereas in our approach, Lemma \ref{K1-surjective} is the last ingredient needed to reprove $K_{1}$-surjectivity of $\calZ$-stable C*-algebras. 

\begin{thm} \label{K1-sur} 
(\cite[Theorem 2]{Jiang}) If $A$ is a C*-algebra, then $A\otimes \calZ$ is $K_{1}$-surjective. 
\end{thm}  

\begin{proof} 
In the unital case, this follows from Theorem \ref{new-pic}, Theorem \ref{K1-surjective} and the fact that $K_{1}$-surjectivity is preserved by inductive limits. For non-unital $A$, since $\calZ$ is $K_{1}$-injective by Theorem \ref{K1-inj}, then $A\otimes \calZ$ is $K_{1}$-surjective by Lemma \ref{non-unital-K1}. 
\end{proof} 

In fact, all we need in the proof is that $\calZ$ is an inductive limit of generalized dimension drop algebras, regardless of the specific connecting maps. 

\section{Homotopy groups for \texorpdfstring{$\calZ$-stable } \ C*-algebras} 
In this section, we prove $K$-stability for both unital and non-unital $\calZ$-stable C*-algebras. In order to cover the non-unital case, we need to extend the definition for unitary groups. For any C*-algebra $A$, we denote the minimal proper unitization of $A$ by $\tilde{A}$, where $\tilde{A} = A\oplus\mathbb{C}$ if $A$ is unital. The canonical maps associated to $\tilde{A}$ are $\iota: A\rightarrow \tilde{A}$ and $q: \tilde{A}\rightarrow \mathbb{C}$. Then the generalized unitary group $\calV(A)=\{u\in \calU(\tilde{A}):q(u)=1\}$ is a topological group and $\calV(A) \cong \calU(A)$ as topological groups when $A$ is unital. We use the notation $\calV_{m}(A)$ for $\calV(M_{m}\otimes A)$, where $m\geq 1$. 

In the unital case, we define $\calU_{\infty}(A)$ to be the inductive limit of groups $\calU_{m}(A)$ with metric preserving connecting maps. Thus $\calU_{\infty}(A)$ is equipped with a well-defined metric from finite stages. For a general C*-algebra, take $\calV_{\infty}(A) = \{u\in \calU_{\infty}(\tilde{A}): q(u) = 1\}$, where $q$ is the matrix amplification of $q: \tilde{A}\rightarrow \mathbb{C}$. Then there are embeddings $i_{m, m'}: \calV_{m}(A)\rightarrow \calV_{m'}(A)$ for $1\leq m\leq m'\leq \infty$ and it is standard to check that maps $\pi_{n}(\calV_{m}(A))\rightarrow  \pi_{n}(\calU_{m}(\tilde{A}))$ induced by inclusions $\iota_{m}: \calV_{m}(A)\rightarrow \calU_{m}(\tilde{A})$ are isomorphisms for $n\geq 1$ and $m\in \mathbb{N}\cup \{\infty\}$. 

By Bott periodicity of $K$-theory (see \cite[Proposition 11.4.1]{Rordam-K} for example), for any $n\geq 0$, the canonical maps induce the following isomorphisms, 
\begin{equation} \label{bott} 
\pi_{n}(\calV_{\infty}(A))\cong 
\begin{cases} 
      K_{0}(A) & \text{if $n$ is odd, } \\
      K_{1}(A) & \text{if $n$ is even. }
\end{cases}
\end{equation} 

The definition for $K$-stability is given by Thomsen using quasi-unitaries in \cite{Thomsen}. Since $\calV(A)$ is homeomorphic to the topological group of quasi-unitaries in $A$ defined in \cite{Thomsen}, the following definition for $K$-stability indeed coincides with Thomsen's. 

\begin{defn} \label{K-stable} 
(\cite[Definition 3.1]{Thomsen}) Let $A$ be a C*-algebra, then it is called \textit{$K$-stable} if for any $m\geq 1$ and $n\geq 0$, the induced maps $\pi_{n}(i_{m, m+1})$ are isomorphisms. 
\end{defn} 

Now we prove the main theorem of the section. 

\begin{thm} \label{main} 
(\cite[Theorem 2.8 and Theorem 3]{Jiang}) Let $A$ be a $\calZ$-stable C*-algebra, then $A$ is $K$-stable. Moreover, for any integer $n\geq 0$, the following isomorphisms are induced by the canonical embedding $i_{1, \infty}: \calV(A)\rightarrow \calV_{\infty}(A)$, 
\begin{equation} \label{final}
\pi_{n}(\calV(A))\cong \begin{cases} 
      K_{0}(A) & \text{if $n$ is odd, } \\
      K_{1}(A) & \text{if $n$ is even. }
   \end{cases}
\end{equation} 
\end{thm} 

\begin{proof} 
When $A$ is $\calZ$-stable, then $A$ is $K_{1}$-injective and $K_{1}$-surjective by Theorem \ref{K1-inj} and Theorem \ref{K1-sur}, which implies $\pi_{0} (\calU(\tilde{A})) \cong K_{1}(\tilde{A}) \cong \pi_{0}(\calU_{\infty}(\tilde{A}))$ by (\ref{bott}). Since $\pi_{0}(\iota_{0})$ and $\pi_{0}(\iota_{\infty})$ are isomorphisms, then $i_{1, \infty}$ induces an isomorphism through 
\begin{equation} \label{main_iso} 
\pi_{0}(\calV(A)) \cong \pi_{0}(\calU(\tilde{A})) \cong \pi_{0}(\calU_{\infty}(\tilde{A})) \cong  \pi_{0}(\calV_{\infty}(A)) = K_{1}(A). 
\end{equation} 
Since $\calZ$-stability of $A$ implies that of $M_{m}\otimes A$, then the maps $i_{m, \infty}$ induce $\pi_{0}(\calV_{m}(A)) \cong K_{1}(A)$ for $m\geq 1$ by the argument above and  thus all the maps $\pi_{0}(i_{m, m+1})$ are isomorphisms. 

For $n\geq 1$, there is a canonical identification between $\pi_{n}(\calV(A))$ and $\pi_{0}(\calV(S^{n}A))$ by \cite[Lemma 2.3]{Thomsen}, where $S^{n}A = C_{0}((0,1)^{n})\otimes A$ is the $n^{th}$ suspension of $A$. When $A$ is $\calZ$-stable, then $S^{n}(M_{m}(A))$ is $\calZ$-stable for any $m, n\geq 1$. Thus by (\ref{main_iso}), the maps $i_{m, \infty}$ induce the following isomorphisms for $m\geq 1$, 
\begin{equation} 
\pi_{n}(\calV_{m}(A)) = \pi_{0}(\calV(S^{n}(M_{m}(A)))) \cong K_{1}((S^{n} (M_{m}(A))) \cong K_{n+1}(A). 
\end{equation} 
Thus $\pi_{n}(i_{m, m+1})$ are isomorphisms for $m\geq 1$ and $n\geq 0$, which implies $K$-stability of $A$. Further by Bott periodicity of $K$-theory, for $m\geq 1$ and $n\geq 0$, 
\begin{equation} 
\pi_{n}(\calV_{m}(A)) \cong 
\begin{cases} 
      K_{0}(A) & \text{if $n$ is odd, } \\
      K_{1}(A) & \text{if $n$ is even. }
\end{cases}
\end{equation} 
In particular, we get (\ref{final}). 
\end{proof}

\bibliographystyle{abbrv} 
\bibliography{references}

\ 

\noindent Shanshan Hua, Mathematical Institute, University of Oxford, Oxford, OX2 6GG, UK. 

\noindent Email address: \texttt{shanshan.hua@maths.ox.ac.uk} 

\end{document}